\documentclass{article}
\usepackage{amsmath}
\usepackage{amssymb}
\usepackage{exscale}
\usepackage[mathscr]{eucal}
\usepackage{latexsym}
\usepackage{placeins}
\usepackage{epsfig}
\usepackage{calc}
\usepackage{theorem}
\usepackage{pifont}      
\usepackage{fancybox}
\usepackage{fullpage}
\usepackage{float}
\usepackage{alltt}  
\usepackage[numbers,sort&compress]{natbib}
\usepackage{psfrag} 

\newcommand{\RR}{\ensuremath{\mathbb R}}

\newcommand{\NN}{\ensuremath{\mathbb N}}

\newcommand{\nnn}{\ensuremath{{n \in \NN}}}

\newcommand{\spann}{\ensuremath{{\operatorname{span}}}}
\newcommand{\dom}{\ensuremath{\operatorname{dom}}}
\newcommand{\halo}{\ensuremath{\operatorname{halo}}}

\newcommand{\gra}{\ensuremath{\operatorname{gra}}}

\newcommand{\ran}{\ensuremath{\operatorname{ran}}}

\newcommand{\Id}{\ensuremath{\operatorname{Id}}}

\newcommand{\bx}{\ensuremath{\mathbf{x}}}

\newtheorem{theorem}{Theorem}[section]

\newtheorem{fact}[theorem]{Fact}
\newtheorem{corollary}[theorem]{Corollary}
\newtheorem{proposition}[theorem]{Proposition}
\newtheorem{definition}[theorem]{Definition}

\theoremstyle{plain}{\theorembodyfont{\rmfamily}
}
\theoremstyle{plain}{\theorembodyfont{\rmfamily}
}
\theoremstyle{plain}{\theorembodyfont{\rmfamily}
}
\theoremstyle{plain}{\theorembodyfont{\rmfamily}
\newtheorem{example}[theorem]{Example}}
\theoremstyle{plain}{\theorembodyfont{\rmfamily}
\newtheorem{remark}[theorem]{Remark}}
\theoremstyle{plain}{\theorembodyfont{\rmfamily}
}

\numberwithin{equation}{section}

\title{Five classes of monotone linear relations and operators}
\author{Mclean Edwards}

\begin{document}
\maketitle

\abstract{The relationships between five classes of monotonicity,
namely $3^*$-, $3$-cyclic, strictly, para-, and maximal monotonicity, are
explored for linear operators and linear relations in Hilbert space.  Where
classes overlap, examples are given; otherwise their
relationships are noted for linear operators in $\RR^2$, $\RR^n$, and
general Hilbert spaces.  Along the way, some results for linear relations are
obtained.}

\section{Introduction}
Monotone operators are multivalued
operators $T : X \to 2^X$ such that for all $x^* \in Tx$ and all $y^* \in Ty$,
\begin{equation}
  \label{e:monotone}
  \langle x - y , x^* - y^* \rangle \geq 0.
\end{equation}

They arise as a generalization of 
subdifferentials of convex functions, and are used extensively in variational inequality (and by reformulation, equilibrium) theory.

Variational inequalities were first outlined in 1966 \cite{StampacchiaHartmann},
and have since been used to model a large number of problems.

\begin{definition}[Variational Inequality Problem]
  Given a nonempty closed convex set $C$ and a monotone operator $T$ acting on $C$,
  the \emph{variational inequality problem}, $VIP(T,C)$, is to
  find an $\bar{x} \in C$ such that for some
  $\bar{x}^* \in T(\bar{x})$
\begin{equation}
    \langle c - \bar{x}, \bar{x}^* \rangle \geq 0
    ~\textrm{for all}~ c \in C.
    \label{e:VIP}
\end{equation}
\end{definition}

They provide a unified framework for, among others, constrained optimization, saddle point,
Nash equilibrium, traffic equilibrium, frictional contact, and complementarity problems.
For a good overview of sample problems and current methods used to solve them, see \cite{FacchineiPang} and \cite{FerrisPang1997}.  

Monotone operators are also important for the theory of partial differential equations,
where monotonicity both characterizes the vector fields of selfdual
Lagrangians \cite{GhoussoubBook} and is crucial for the determination of
equilibrium solutions (using a variational inequality) for elliptical
and evolution differential equations and inclusions (see for instance \cite{AttouchEvolution}).

Over the years, various classes of monotone operators have been introduced in the exploration of their theory, however there have been few attempts to comprehensively compare those in use across disciplines.

Five special classes of monotone operators are studied here:
strictly monotone, $3$-cyclic monotone, $3^*$-monotone, paramonotone and maximal monotone.
All possible relationships between these five properties are
explored for linear operators in $\RR^2$, $\RR^n$, and in general Hilbert
space, and the results are summarized in Tables~\ref{table:linear2D} and
\ref{table:linear} and in 
Figures~\ref{fig:4compareL}, \ref{fig:4compareLR2}, and \ref{fig:4compareLRn}.

\begin{definition}[paramonotone]
    \label{d:PM}
    An operator $T: X \to 2^X$ is said to be \emph{paramonotone}
    if $T$ is monotone and for $x^* \in Tx, y^* \in Ty$,
    $\langle x - y, x^* - y^* \rangle = 0$ implies that
    $x^* \in Ty$ and $y^* \in Tx$.
\end{definition}

A number of iterative methods for solving (\ref{e:VIP}) have required paramonotonicity to converge.
Examples include an interior point method using Bregman functions \cite{CensorIusemZenios},
an outer approximation method \cite{BurachikJurandirSv}, 
and proximal point algorithms \cite{AuslenderHaddou95} \cite{BurachikIusem98}.
Often, as in \cite{CruzIusem2010}, with more work it is possible to show convergence with paramonotonicity where previously stronger conditions, such as strong monotonicity, were required.
Indeed, the condition first emerged in this context \cite{Bruck75} as a sufficient condition for the convergence of a projected-gradient like method.
For more on the theory of paramonotone operators and why this condition is
important for variational inequality problems, 
see \cite{Iusem98} and \cite{HSD2004}.

\begin{definition}[strictly monotone]
    \label{d:SM}
    An operator $T: X \to 2^X$ is said to be \emph{strictly monotone} 
    if $T$ is monotone and for all $(x,x^*), (y,y^*) \in \gra T$,
    $\langle x - y, x^* - y^* \rangle = 0$ implies that
    $x = y$.
\end{definition}

Strict monotonicity is a stronger condition than paramonotonicity (Fact~\ref{f:SMisPM}),
and the strict monotonicity of an operator $T$ guarantees the uniqueness of a solution to the variational
inequality problem (see for instance \cite{FacchineiPang}).
These operators are somewhat analogous to the subdifferentials of strictly convex functions.

We adopt the notation of \cite{ZeidlerIIB} 
and use the term $3^*$-monotone, although this property was first introduced with no name.
The property was first referenced simply by ``$*$'' \cite{BrezisHaraux76} by Br\'ezis and Haraux, and such operators were sometimes called (BH)-operators \cite{Chu96} in honour of these original authors. 
More recently the property has also taken on the name ``rectangular''
since the domain of the Fitzpatrick function of a monotone operator is rectangular
precisely when the operator is $3^*$ monotone \cite{SimonsLC1}.

\begin{definition}[$3^*$ monotone]
  \label{d:3starmonotone}
  An operator $T:X \to 2^X$ is said to be $3^*$-monotone
  if $T$ is monotone and
  for all $z$ in the domain of $T$ and for all $x^*$ in the range of $T$
  \begin{equation}
    \sup_{(y,y^*)\in\gra T} \langle z-y , y^*-x^* \rangle < +\infty.
    \label{e:3starmonotone}
  \end{equation}
\end{definition}

$3^*$-monotonicity has the important property in that
if $T_1$ and $T_2$ are $3^*$-monotone, then the sum of their ranges is the range of their sum.
For instance,
if two operators are $3^*$-monotone, and one is surjective,
then if the sum is maximal monotone it is also surjective.
Furthermore, if both are continuous monotone linear operators,
and at least one is $3^*$-monotone, then the kernel of the sum
is the intersection of the kernels \cite{Bauschke44}.
This property can be used, as shown in \cite{BrezisHaraux76},
to determine when solutions to $T^{-1}(0)$ exist by demonstrating
that $0$ is in the interior (or is not in the closure) of the sum of the ranges of an 
intelligent decomposition of a difficult to evaluate maximal monotone operator.
It has also been shown for linear relations on Banach spaces
that $3^*$-monotonicity guarantees the existence of solutions to the primal-dual problem pairs in \cite{PennanenThesis}. It should also be noted that operators with bounded range \cite{ZeidlerIIB} and strongly coercive operators \cite{BrezisHaraux76} are $3^*$-monotone.

\begin{definition}[$n$-cyclic monotone] \label{d:ncyclicmonotone}
    Let $n \geq 2$.  An operator $T: X \to 2^X$ 
    is said to be $n$-cyclic monotone if
    \begin{equation}
			\label{e:nCM}
        \left.
        \begin{array}{rcl}
            (x_1, x_1^*) && \in \gra T \\
            (x_2, x_2^*) && \in \gra T \\
            \cdots && \in \gra T \\
            (x_n, x_n^*) && \in \gra T\\
            x_{n+1} &=& x_1
        \end{array}
        \right\}
        \Rightarrow
        \sum_{i=1}^n \langle x_i - x_{i+1}, x_i^* \rangle \geq 0
    \end{equation}
  A \emph{cyclical monotone} operator is one that is $n$-cyclic monotone for all
  $n \in \NN$.  
\end{definition}

Note that 2-cyclic monotonicity is equivalent to monotonicity.
By substituting $(a_n, a_n^*) := (a_1, a_1^*)$, 
it easily follows from the definition that 
any $n$-cyclic monotone operator is $(n-1)$-cyclic monotone.
$1$-cyclic monotonicity is not defined, since the $n=1$ case for
(\ref{e:nCM}) is trivial.  
$3$-cyclic monotone operators serve to represent a special case of $n$-cyclic monotone operators
that is also a stronger condition than $3^*$-monotonicity.  Of note, all
subdifferentials of convex functions are cyclical monotone
\cite{RockWets}.

\begin{definition}[maximality]
  An operator is \emph{maximal $n$-cyclic monotone}
  if its graph cannot be extended while preserving $n$-cyclic monotonicity.
  A \emph{maximal monotone} operator is a maximal $2$-cyclic monotone operator.
  A \emph{maximal cyclical monotone} operator is a cyclical
  monotone operator such that all proper graph extensions are not cyclical monotone.
\end{definition}

There is a rich literature on the theory (see \cite{BorweinConvex} for a good overview)
and application (for instance \cite{EcksteinBertsekas1992})
of maximal monotone operators.  Furthermore, it is well known that a maximal monotone operator $T$ has the property that
$T^{-1}(0)$ is convex, a property shared by paramonotone operators with convex domain (Proposition~\ref{p:Tinverse0PM}),
and analogous to the fact that the minimizers of a convex function form a convex set.
Maximal monotonicity is also an important property for general differential inclusions
\cite{Papageorgiou1998} \cite{Staicu1994}.

\begin{definition}[Five classes of monotone operator]
    An operator $T: X \to 2^X$ is said to be [Class] 
    (with abbreviation [Code]) 
    if and only if $T$ is monotone and for every 
    $(x,x^*), (y,y^*), (z,z^*)$ in \gra T
    one has [Condition].\\
  \begin{tabular}{l l l}
    \label{table:definitions}
      Code & Class & Condition (A) \\
      \hline
      & monotone & $\langle x-y, x^* - y^* \rangle \geq 0$ \\
      PM & paramonotone  & $\langle x - y, x^* - y^*\rangle = 0 \Rightarrow (x,y^*), (y,x^*) \in \gra T$ \\
      SM & strictly monotone  & $\langle x - y, x^* - y^*\rangle = 0 \Rightarrow x = y$ \\
      3CM & $3$-cyclic monotone  & $\langle x - y, x^*\rangle + \langle y-z, y^*\rangle + \langle z - x, z^* \rangle \geq 0 $\\
      MM & maximal monotone  & $(\forall a \in X) (\forall a^* \in X)$ \\
          & $\quad \langle x-a, x^* - a^* \rangle \geq 0 \Rightarrow (x,x^*) \in \gra T$\\
      3* & $3^*$-monotone  & $\sup_{(a,a^*)\in\gra T} \langle z-a , a^*-x^* \rangle < +\infty $ 
  \end{tabular}
\\
The order above, PM-SM-3CM-MM-3*, is fixed to allow a binary code representation of the classes to which an operator belongs.  For instance, an operator with the code 10111 is paramonotone, not strictly monotone, $3$-cyclic monotone, maximal monotone, and $3^*$-monotone.
\end{definition}

After noting some general relationships between these classes in Section~\ref{s:prelim},
we note in Section~\ref{s:productSpace} that monotone operators belonging to particular combinations of these classes can be constructed in a product space.

Linear relations are a multi-valued extension of linear operators, and are
defined by those operators whose graph forms a vector space.  This is a
natural extension to consider as monotone operators are often multivalued.  We
consider linear relations in Section~\ref{s:LR}, and explore their
characteristics and structure.
Of particular note, we fully explore the manner in which linear relations can be multivalued and remark on a curious property of linear relations whose domains are not closed.
Finally, we obtain a generalization to the fact that bounded linear operators
that are $3^*$-monotone are also paramonotone (a corollary to a result in
\cite{BrezisHaraux76}), with conditions different from those in
\cite{HXLtoappear}, and demonstrate by example that there is $3^*$-monotone linear relation that is not paramonotone.

In Section~\ref{s:LRPM}, 
we list various examples of linear
operators satisfying or failing to satisfy the 5 properties defined above.
The examples are chosen to have full domain, low dimension, and be continuous where possible.
This is shown to yield a complete
characterization of the dependence or independence of these five classes of
monotone operator in $\RR^2$, $\RR^n$, and in a general Hilbert space $X$.  
One result of this section is that paramonotone and linear operators in $\RR^2$ are exactly the
symmetric or strictly monotone operators in $\RR^2$.

We assume throughout that $X$ is a real Hilbert space,
with inner product $\langle \cdot, \cdot \rangle$.
When an operator $T : X \to 2^X$ is such that for all $x \in X$,
$Tx$ contains at most one element, such operators are called \emph{single-valued}.
When $T$ is single-valued, for brevity $Tx$ is at times considered as a point rather than as a set (ie: $x^* \in Tx$).
The \emph{orthogonal complement} of a set $C \subset X$
is denoted by $C^{\perp}$ and defined by
\begin{equation}
  C^{\perp} := \{x \in X : \langle x, c \rangle = 0 ~\forall c \in C \}
  \label{e:perp}
\end{equation}
Note that for any set $C \subset X$, the set $C^{\perp}$ is closed in $X$.
$P_V$ denotes the metric projection where $V$ is a closed subspace of $X$.
We use the convention that for set addition $A + \emptyset = \emptyset$,
where $\emptyset$ is the empty set.
A monotone extension $\tilde{T} : X \to 2^X$ of a monotone operator $T: X \to 2^X$ is a monotone operator
such that $\gra T \subsetneq \gra \tilde{T}$,
where $\gra T := \{(x,x^*) : x \in \dom T, x^* \in Tx \}$.
A selection of an operator $T : X \to 2^X$ is an operator $\tilde{T}$ such that $\gra \tilde{T} \subset \gra T$,
and a single-valued selection of $T$ is such an operator $\tilde{T}$ where $\tilde{T} : X \to X$.

\section{Preliminaries} \label{s:prelim}

The following arises from the definition of strict monotonicity 
and paramonotonicity.

\begin{fact}
  \label{f:SMisPM}
  Any strictly monotone operator $T : X \to 2^X$
  is also paramonotone.
\end{fact}

Two synonymous definitions
of $3$-cyclic monotonicity are worth explicitly stating.
For an operator $T: X \to 2^X$ to be $3$-cyclic monotone,
every 
$(x,x^*),(y,y^*),(z,z^*) \in \gra T$ must satisfy
\begin{equation} 
  \label{e:3cyclicgeq0}
  \langle x-y, x^* \rangle + \langle y-z, y^* \rangle + \langle z-x, z^* \rangle \geq 0,
\end{equation}
or equivalently 
\begin{equation}
  \langle z - y, y^* - x^* \rangle \leq \langle x - z, x^* - z^*\rangle.
  \label{e:3cm2ineq}
\end{equation}
From (\ref{e:3cm2ineq}), the following fact is obvious.

\begin{fact}
  \label{f:3CMis3*}
  Any $3$-cyclic monotone operator $T : X \to 2^X$
  is also $3^*$-monotone.
\end{fact}

Another relationship between these classes
of monotone operator was discovered in 2006.

\begin{proposition} \label{p:3-2cyclic} \cite{Hadjisavvas06}
    If $T$ is $3$-cyclic monotone and maximal ($2$-cyclic) monotone,
    then $T$ is paramonotone.
\end{proposition}
\begin{proof}
Suppose that for some choice of $(x, x^*), (y, y^*) \in \gra(T)$,
$\langle x - y, x^* - y^* \rangle = 0$,
so 
$\langle y - x, x^* \rangle = \langle y - x, y^* \rangle$.
Since $T$ is $3$-cyclic monotone, every $(z, z^*) \in \gra(T)$ satisfies
\begin{eqnarray*}
    0 &\geq& 
    \langle y - x, x^* \rangle + \langle z - y, y^* \rangle
    + \langle x - z, z^* \rangle \\
    &=& \langle - x, y^* \rangle
	+\langle z, y^* \rangle + \langle x - z, z^* \rangle \\
    &=& \langle z - x, y^* \rangle + \langle x - z, z^* \rangle \\
    &=& \langle x - z, z^* - y^* \rangle
\end{eqnarray*}
and so
\[\langle x - z, y^* - z^* \rangle \geq 0 \quad \forall(z, z^*) \in \gra(T). \]
Since $T$ is maximal monotone, $y^* \in Tx$.
By exchanging the roles of $x$ and $y$ above, 
it also holds that $x^* \in T(y)$, 
and so $T$ is paramonotone.
\end{proof}


When finding the zeros of a monotone operator, it can be useful to know if the solution set is convex or not.  It is well known that for a maximal monotone operator $T$, $T^{-1}(0)$ is a closed convex set
(see for instance \cite{BauschkeBook}).  
A similar result also holds for paramonotone operators.

\begin{proposition}
  \label{p:Tinverse0PM}
  Let $T:X \to 2^X$ be a paramonotone operator with convex domain.
  Then $T^{-1}(0)$ is a convex set.
\end{proposition}
\begin{proof}
	Suppose $T^{-1}(0)$ is nonempty.
  Let $x,y,z \in X$ such that $0 \in Tx$, $0 \in Tz$, and
  $y = \alpha x + (1-\alpha)z$ for some $\alpha \in ]0,1[$.
  Then, $x-y = (1-\alpha)(x-z)$ and
  $y-z = \alpha (x-z)$, so $x-y = \frac{1-\alpha}{\alpha} (y-z)$.
  Since $T$ has convex domain, $Ty \neq \emptyset$.
  By the monotonicity of $T$, for all $y^* \in Ty$
  \[0 \leq \langle x-y, -y^* \rangle = \frac{1-\alpha}{\alpha} \langle y-z, -y^* \rangle
  \qquad \textrm{and} \qquad
  0 \leq \langle y - z, y^* \rangle, \]
  and so $\langle y-z, y^* \rangle = 0$.
  Therefore, by the paramonotonicity of $T$, $0 \in T(y)$, and so the set
  $T^{-1}(0)$ is convex.
\end{proof}

However, if an operator is not maximal monotone, there is no guarantee that $T^{-1}(0)$ is closed,
even if paramonotone, as the operator $T : \RR \to \RR$ below demonstrates:
\begin{equation}
  \label{e:PMwithoutclosedinverse}
  Tx := \left\{ \begin{array}{ll}
    -1, & x \leq -1, \\
    0, & x \in ]-1,1[, \\
    1, & x \geq 1.
  \end{array} \right. 
\end{equation}

\section{Monotone operators on product spaces}
\label{s:productSpace}

Let $X_1$ and $X_2$ be Hilbert spaces,
and consider set valued operators
$T_1: X_1 \to 2^{X_1}$ and $T_2: X_2 \to 2^{X_2}$.
The product operator
$T_1 \times T_2 : X_1 \times X_2 \to 2^{X_1 \times X_2}$
is defined as
$(T_1 \times T_2) (x_1, x_2) :=
\{(x_1^*, x_2^*)$ : $x_1^* \in T_1 x_1$
and $x_2^* \in T_2 x_2$ \}.

\begin{proposition}
  \label{p:productM}
    If both $T_1$ and $T_2$
    are monotone, then
    the product operator 
    $T_1 \times T_2$ 
    is also monotone.
\end{proposition}
\begin{proof}
  For any points
  $\left( (x_1,x_2), (x^*_1, x^*_2) \right),
  \left( (y_1,y_2), (y^*_1, y^*_2) \right) \in \gra (T_1 \times T_2)$,
  \[ 
  \begin{array}{cl}
    & \langle (x_1,x_2) - (y_1,y_2), (x^*_1, x^*_2) - (y^*_1, y^*_2) \rangle \\
    =& \langle x_1 - y_1, x_1^* - y_1^* \rangle
      + \langle x_2 - y_2, x_2^* - y_2^*\rangle \geq 0.
  \end{array}
  \]
\end{proof}

\begin{proposition}
    \label{p:productparamonotone}
    If both $T_1$ and $T_2$
    are paramonotone, then
    the product operator 
    $T_1 \times T_2$ 
    is also paramonotone.
\end{proposition}
\begin{proof}
If 
$x_i^* \in T_i x_i$, $y_i^* \in T_i y_i$
for $i \in \{1,2\}$ and
\[
\langle (x_1, x_2) - (y_1, y_2), 
(x_1^*, x_2^*) - (y_1^*, y_2^*) \rangle = 0, \]
then
$\langle x_i - y_i, x_i^* - y_i^* \rangle = 0$ for $i \in \{1,2\}$
since both $T_1$ and $T_2$ are monotone.
By the paramonotonicity of $T_1$ and $T_2$,
$y_i^* \in T_i x_i$ and $x_i^* \in T_i y_i$
for $i \in \{1,2\}$,
and so $(x_1^*, x_2^*) \in T_1 \times T_2 (y_1, y_2)$
and $(y_1^*, y_2^*) \in T_1 \times T_2 (x_1, x_2)$.
\end{proof}

By following the same proof structure as Proposition~\ref{p:productparamonotone}, a similar result immediately follows for some other monotone classes.

\begin{proposition}
    \label{p:productSMCM3*}
    If both $T_1$ and $T_2$ belong to the same monotone class, 
    where that class is one of strict, $n$-cyclic, or $3^*$-monotonicity,
    then so does their product operator $T_1 \times T_2$.
\end{proposition}

\begin{proposition}
    \label{p:productMM}
    If both $T_1$ and $T_2$
    are maximal monotone, then
    the product operator 
    $T_1 \times T_2$ 
    is also maximal monotone.
\end{proposition}
\begin{proof} 
	Suppose $T_1 \times T_2$ is not maximal monotone.
	Then there exists a point
	$\left( (x_1 , x_2), (x_1^*, x_2^*) \right) \notin \gra (T_1 \times T_2)$
  such that for all 
	$\left( (y_1, y_2), (y_1^*, y_2^*) \right) \in \gra (T_1 \times T_2)$
  \begin{equation}
    \label{e:productMM}
	\langle x_1 - y_1, x_1^* - y_1^* \rangle
  + \langle x_2 - y_2, x_2^* - y_2^*\rangle \geq 0,
  \end{equation}
	and at least one of $(x_1,x_1^*) \notin \gra T_1$
	or $(x_2,x_2^*) \notin \gra T_2$.
	Suppose without loss of generality that
	$(x_1,x_1^*) \notin \gra T_1$. \\
	By the maximality of $T_1$, $\langle x_1 - z_1, x_1^* - z_1^* \rangle < 0$ for some
	$(z_1,z_1^*) \in \gra T_1$, 
  and so by setting $(y_1, y^*_1) := (z_1, z^*_1)$ in (\ref{e:productMM}),
  $\langle x_2 - y_2, x^*_2 - y^*_2 \rangle \geq 0$
  for all $(y_2,y^*_2) \in \gra T_2$.
  Since $T_2$ is maximal monotone, it must be that
	$(x_2,x_2^*) \in \gra T_2$.  \\
  Clearly,
	$\left( (z_1,  x_2), (z_1^*,  x_2^*) \right) \in \gra T_1 \times T_2$,
	yet
	\[\langle (x_1, x_2) - (z_1, x_2), 
	(x_1^*, x_2^*)  - (z_1^*, x_2^*) \rangle < 0.\]
  This is a contradiction of $(\ref{e:productMM})$,
	and so $T_1 \times T_2$ is maximal monotone.
\end{proof}

Of course, if an operator $T_1 : X \to 2^X$ fails to satisfy the conditions for any of the classes of monotone operator
here considered, then the space product of that operator with any other operator
$T_2 : Y \to 2^Y$, namely $T_1 \times T_2 : X \times Y \to 2^{X \times Y}$, will also fail the same condition.
Simply consider the set of points $P$ in the graph of $T_1$ which violate a particular condition in $X$, and instead
consider the set of points $\tilde{P} := \{(p,a) \times (p^*,a^*): p \in P\}$ for a fixed arbitrary point $(a,a^*) \in \gra T_2$.
Clearly $\tilde{P} \subset \gra T_1 \times T_2$, and this set will violate the same conditions in $X \times Y$
that $P$ violates for $T_1$ in $X$.  For instance,
\[ \langle (w, a) - (x, a), (y^*,a^*) - (z^*,a^*) \rangle
= \langle w - x, y^* - z^* \rangle. \]
In this manner, the lack of a monotone class property
(be it $n$-cyclic, para-, maximal, $3^*$-, nor strict monotonicity) is dominant in the product space.

  Taken together, the results of this section are that the product operator $T_1 \times T_2$ of monotone operators $T_1$ and $T_2$ operates with respect to monotone class inclusion as a logical AND operator applied to the monotone classes of $T_1$ and $T_2$.
  For instance, suppose that $T_1$ is paramonotone, not strictly monotone, $3$-cyclic monotone, maximal monotone, and $3^*$-monotone (with binary representation $10111$),
  and suppose that $T_2$ is paramonotone, strictly monotone, not $3$-cyclic monotone, maximal monotone, and not $3^*$-monotone (with binary representation $11010$).
  Then, $T_1 \times T_2$ is paramonotone, not strictly monotone, not $3$-cyclic monotone, maximal monotone, and not $3^*$-monotone (with binary representation $10010$).

\section{Linear Relations} \label{s:LR}
Using the nomenclature of R. Cross \cite{RCross},
we define linear relations, which are set-valued generalizations of linear operators.

\begin{definition}[linear relation]
  An operator $A:X \to 2^X$ is a \emph{linear relation} if 
  $\dom A$ is a linear subspace of $X$ and 
  for all $x,y \in \dom A$, $\lambda \in \RR$
  \begin{enumerate}
  \item $\lambda Ax \subset A(\lambda x)$,
  \item $Ax + Ay \subset A(x+y)$.
  \end{enumerate}
\end{definition}

Equivalently, linear relations are exactly those operators $T : X \to 2^X$ whose graphs are linear subspaces of $X \times X$.  The following results on linear relations are well known.

\begin{fact}
  \cite{Yagi91}
  \label{f:lineartraits}
  For any linear relation
  $A:X \to 2^X$,
  \begin{enumerate}
  \item $\lambda Ax = A(\lambda x)$ for all $x \in \dom A$, $0 \neq \lambda \in \RR$,
    \label{lineartraits1}
  \item $Ax + Ay = A(x+y)$ for all $x,y \in \dom A$,
    \label{lineartraits2}
  \item $A0$ is a linear subspace of $X$,
    \label{lineartraits3}
  \item $Ax = x^* + A0$ for all $(x,x^*) \in \gra A$,
    \label{lineartraits4}
  \item If $A$ is single valued at any point, it is single valued at every point in its domain.
    \label{lineartraits5}
  \end{enumerate}
\end{fact}

\begin{proposition}
  \label{p:lineartraitA0perp}
  Suppose $A : X \to 2^X$ is a linear relation, and let $x \in \dom A$.
  Then,
  $P_{A0^{\perp}} Ax$ is a singleton and
  \begin{equation}
    Ax \subset P_{A0^{\perp}} Ax + \overline{A0}.
    \label{e:lineartraitA0perpsubset}
  \end{equation}
  If $A0$ is closed, 
  then
  there is a unique
  $x_0^* \in Ax$ such that
  $x_0^* \in A0^{\perp}$, where
  $x_0^* = P_{A0^{\perp}} x^*$ for all $x^* \in Ax$.
\end{proposition}
\begin{proof}
  Let $x \in \dom A$.
  Since $\overline{A0}$ and $A0^{\perp}$ are closed subspaces
  such that $\overline{A0} + A0^{\perp} = X$, then for all $x^* \in X$,
  $x^* = P_{\overline{A0}} x^* + P_{A0^{\perp}} x^*$.
  By Fact~\ref{f:lineartraits} \ref{lineartraits4},
  (\ref{e:lineartraitA0perpsubset}) holds and
  $P_{A0^{\perp}} Ax$ is a singleton.
  If $A0$ is closed,
  then for all $x^* \in Ax$,
  \[
  Ax = x^* + A0 = P_{A0^{\perp}}x^* + A0.
  \]
  Therefore,
  $ P_{A0^{\perp}}y^* = P_{A0^{\perp}}x^*$
  for all $y^* \in Ax$.
  Furthermore, 
  since $0 \in A0$ always,
  $P_{A0^{\perp}}x^* \in Ax$. \\
\end{proof}

\begin{proposition}
    \label{p:linearMM}
    Any monotone linear relation $A: X \to 2^X$
    with full domain is maximal monotone
    and single valued.  
\end{proposition}
\begin{proof}
  Suppose that $A :X \to 2^X$ is a linear relation where $\dom A = X$.
  Let $(z, z^*)$ be a point such that 
  $\langle z - y, z^* - y^* \rangle \geq 0$ for all $(y,y^*) \in \gra A$.
  Choose an arbitrary $z_0^* \in Az$.
  Let $y = z - \varepsilon x$ for arbitrary $(x,x^*) \in \gra A$ and $\varepsilon > 0$,
  so that by linearity $-\varepsilon x^* \in A(-\varepsilon x)$.
  Therefore $z_0^* - \varepsilon x^* \in Ay$ and so
  $\langle \varepsilon x, z^* - z_0^* + \varepsilon x^* \rangle \geq 0$.
  Divide out the $\varepsilon$,
  and send $\varepsilon \to  0^+$
  so that 
  $\langle x, z^* - z_0^* \rangle \geq 0$ for all $x \in X$.
  Hence $z^* = z_0^*$ and $T$ is single valued and maximal monotone.
\end{proof}

\begin{proposition} \cite{Bauschke55}
  \label{p:LRdomain0}
  If $A : X \to 2^X$ is a monotone linear relation, 
  then $\dom A \subset (A0)^{\perp}$
  and $A0 \subset (\dom A)^{\perp}$.
\end{proposition}

\begin{corollary} \cite{Bauschke55}
  \label{c:LRMMb55}
  If a linear relation $A : X \to 2^X$ is maximal monotone, then
  $(\dom A)^{\perp} = A0$, and so 
  $\overline{\dom A} = (A0)^{\perp}$ and $A0$ is a closed subspace.
\end{corollary}

This leads to a partial converse result to Proposition~\ref{p:linearMM}.

\begin{corollary} \label{c:LRMMfulldomain}
  If a maximal monotone single-valued linear relation $A: X \to X$
  is locally bounded, then it has full domain.
\end{corollary}
\begin{proof}
  Since $A$ is single valued, $A0 = 0$, and so by Corollary~\ref{c:LRMMb55},
  $\overline{\dom A} = (A0)^{\perp} = X$.
  Choose any point $x \in X$.  Since $\dom A$ is dense in $X$,
  there exist a sequence $(y_n,y^*_n)_{\nnn} \subset \gra A$ 
  such that $y_n \to x$.  Since $A$ is locally bounded, a subsequence
  $(y_{\phi(n)}^*)_{\nnn}$ of $(y_n^*)_{\nnn}$ weakly converges to some point $x^* \in X$.
  Therefore, for all $(z,z^*) \in \gra A$,
  \[
  0 \leq \lim_{n \to +\infty} \langle y_{\phi(n)} - z, y_{\phi(n)}^* - z^* \rangle
  = \langle x - z, x^* - z^* \rangle.
  \]
  Since $A$ is maximal monotone, $(x,x^*) \in \gra A$, and so
  $A$ has full domain.
\end{proof}

The following fact appears in Proposition~2.2 in \cite{Bauschke55}.

\begin{fact}[\cite{Bauschke55}]
  \label{f:LRsimpleip}
  Let $A : X \to 2^X$ be a monotone linear relation.
  For any $x,y \in \dom A$, the set
  \[ \{ \langle y, x^* \rangle : x^* \in Ax \} \]
  is a singleton, the value of which can be denoted simply by
  $ \langle y, Ax \rangle$.
\end{fact}
\begin{proof}
  Let $x,y \in \dom A$ and
  suppose that $x_1^*, x_2^* \in Ax$.
  By Fact~\ref{f:lineartraits} \ref{lineartraits4},
  $x_2^* - x_1^* \in A0$.
  Now, by Proposition~\ref{p:LRdomain0},
  $A0 \subset (\dom A)^{\perp}$,
  and so $x_2^* - x_1^* \in (\dom A)^{\perp}$.
  Since $y \in \dom A$,
  $\langle y, x_1^* \rangle = \langle y, x_2^* \rangle$.
\end{proof}

Proposition~\ref{p:LRdimreduction} below demonstrates that multi-valued linear relations are
closely related to a number of single-valued linear relations.  Note especially that
$V = A0^{\perp}$ and $V = \overline{dom A}$ both satisfy the conditions below.

\begin{proposition}[dimension reduction]
  \label{p:LRdimreduction}
  Suppose that $A : X \to 2^X$ is a monotone linear relation.
  Let $V \subset X$ satisfy
  \begin{enumerate}
    \item $V$ is a closed subspace of $X$,
    \item $\dom A \subset V$, and
    \item $A0 \subset V^\perp$.
  \end{enumerate}
  Then the operator 
  $\tilde{A} : V \to 2^V$,
  defined by $\tilde{A}x := P_{V} Ax$ on $\dom A$,
  where $\dom A = \dom \tilde{A}$,
  is a single-valued monotone linear relation.
  In the case where $V = A0^{\perp}$ and $A0$ is closed, the operator $\tilde{A}$
  is a single-valued selection of $A$.
  If $A$ is maximal monotone, then $V = A0^{\perp} = \overline{\dom A}$ is the only subspace satisfying conditions (i)--(iii) above, and $\tilde{A}$ is a maximal monotone single-valued selection of $A$.
\end{proposition}
\begin{proof}
  For any $x \in X$,
  $P_V(x) = P_V(P_{A0^{\perp}} x + P_{\overline{A0}} x) = P_V(P_{A0^{\perp}} x)$
  as $\overline{A0} \subset V^{\perp}$.
  By Proposition~\ref{p:lineartraitA0perp},
  $\tilde{A}$ is always single-valued,
  and if $A0$ is closed,
  $P_{A0^{\perp}} x^* \in Ax$
  for each $(x,x^*) \in \gra A$,
  and so if $V = A0^{\perp}$, 
  then $\tilde{A}$ is a selection of $A$.
  Consider now arbitrary $(y,\tilde{y}^*), (z,\tilde{z}^*) \in \gra \tilde{A}$,
  and $\lambda \in \RR$.
  Then, for $y^* \in Ay$ and $z^* \in Az$,
  we have that $P_V y^* = \tilde{y}^*$ and $P_V z^* = \tilde{z}^*$.
  Since $A$ is a linear relation,
  $(y+\lambda z, y^* + \lambda z^*) \in \gra A$.
  Therefore,
  $(y + \lambda z, P_V (y^* + \lambda z^*) ) \in \gra \tilde{A}$,
  and since $P_V$ is itself a linear operator,
  $P_V (y^* + \lambda z^*) = \tilde{y}^* + \lambda \tilde{z}^*$,
  it follows that 
  $\tilde{y}^* + \lambda \tilde{z}^* \in \tilde{A}(y + \lambda z)$
  Since $\dom A = \dom \tilde{A}$,
  the operator $\tilde{A}$ is a linear relation.
  Finally, suppose that $A$ is maximal monotone,
  and so from Corollary~\ref{c:LRMMb55} we have that
  $A0^{\perp} = \overline{\dom A}$ and $A0$ is closed.
  The only subspace $V$ satisfying the
  conditions in this case is $V = A0^{\perp}$.
  Suppose there exists a point $(x,x^*)$
  where $x \in V = A0^{\perp}$,
  that is monotonically related to $\gra \tilde{A}$.
  For all $(z,z^*) \in \gra A$,
  there is a $y \in A0$ such that 
  $y + P_V z^* = z^*$.
  Then, by Fact~\ref{f:lineartraits} \ref{lineartraits4},
  \[
  0 \leq
  \langle x - z, x^* - z^* \rangle
  = \langle x - z, x^* - y - P_V z^* \rangle 
  = \langle x - z, x^* - P_V z^* \rangle .
  \]
  Therefore, $(x,x^*)$ also extends $A$, and 
  since $A$ is maximal monotone, $(x,x^*) \in \gra A$.
  Since $x^* \in V$, $P_V x^* = x^*$ and so $(x,x^*) \in \gra \tilde{A}$.
  Therefore, $\tilde{A}$ is maximal monotone.
\end{proof}

From the results in this section so far, we know that monotone linear relations $A : X \to 2^X$ 
can only be multivalued such that $A0$ is a subspace of $X$, $Ax = x^* + A0$ for any $x^* \in Ax$,
and $A0 \subset (\dom A)^{\perp}$.  For the purposes of calculation by the inner product,
for any $x,z \in \dom A$,
\begin{equation}
  \langle x, Az  \rangle = \langle x, \tilde{A} z \rangle
\end{equation}
where $\tilde{A}$ is the single-valued operator (a selection of $A$ if $A0$ is closed) as calculated in
Proposition~\ref{p:LRdimreduction} for $V = A0^{\perp}$.
In the other direction, any single-valued monotone linear relation $\tilde{A} : X \to 2^X$
can be extended to a multi-valued monotone linear relation $A : X \to 2^X$ by choosing any subspace 
$V \subset (\dom A)^{\perp}$ and setting $Ax := \tilde{A}x + V$.

Now, in the unbounded linear case, maximal monotone operators may not have a closed domain.  The concept of a \emph{halo} well captures this aspect.

\begin{definition}[halo]
  The \emph{halo} of a monotone linear relation $A : X \to 2^X$ is the set
  \begin{equation}
    \halo A := \left\{ x \in X : (\exists M)(\forall (y,y^*) \in \gra A)
    \langle x - y, y^* \rangle \leq M \|x- y\| \right\}.
    \label{e:halo}
  \end{equation}
\end{definition}

\begin{fact}\cite{Bauschke55}
  If $A : X \to 2^X$ is a monotone linear relation, then 
  $\dom A \subset \halo A \subset (A0)^{\perp}$.
  Furthermore, $A$ is maximal monotone if and only if 
  $A0^{\perp} = \overline{\dom A}$ and $\halo A = \dom A$.
\end{fact}

Now, if the domain of a linear relation is not closed, we have the following curious result.
Below, $A^m$ denotes the iterated operator composition, where for instance $A^3x = A(A(Ax))$.
Note that if $\dom A$ is dense in $X$, the operator $P_V A$ is the same as $A$.

\begin{proposition}
  \label{p:curiousLRresult}
  Suppose a maximal monotone linear relation $A : X \to 2^X$ is such that $\dom A$ is not closed,
  and let $V := \overline{\dom A}$.
  Then, there is a sequence $(z_n)_{\nnn} \subset \dom A$
  such that
  \begin{eqnarray}
    (P_V A)^m(z_n) \in \dom A, & \forall 1 \leq m < n \\
    (P_V A)^n(z_n) \notin \dom A, &
    \label{e:curiousLRresult}
  \end{eqnarray}
  where for all $z \in \dom A$, $P_V Az$ is a singleton set.
\end{proposition}
\begin{proof}
  Since $A$ is maximal monotone, $\dom A = \halo A \subsetneq \overline{\dom A}$,
  and by Corollary~\ref{c:LRMMb55},
  $V = A0^{\perp}$.
  Therefore, by Proposition~\ref{p:lineartraitA0perp},
  $P_V Az \subset Az$ and is a singleton for every $z \in \dom A$.
  Choose any point $z_0 \in V$ such that $z_0 \notin \dom A$. 
  We shall generate the sequence $(z_n)_{\nnn} \subset \dom A$ iteratively as follows.
  For some $n \geq 0$, suppose that $z_n \in V$.
  By Minty's theorem \cite{Minty1962}, 
  since $A$ is maximal monotone,
  $\ran(\Id + A) = X$.
  Therefore, there exists a
  $z_{n+1} \in \dom A$ such that
  $z_n \in z_{n+1} + A z_{n+1}$.
  Since $z_n, z_{n+1} \in V$,
  $z_n \in z_{n+1} + P_V A z_{n+1}$,
  and so as $P_V Az_{n+1}$ is a singleton, 
  \[ P_V A z_{n+1} = \{ z_n - z_{n+1} \} . \]
  Now, since both $P_V$ and $A$ are linear operators, if $n \geq 2$
  \begin{equation}
    \begin{array}{lcl}
   (P_V A)^2 z_{n+1} &=&   P_V A (z_n - z_{n+1})    \\
   &=&  P_V A z_n - P_V A z_{n+1}  \\
   &=&  \{ z_{n-1} - 2z_n + z_{n+1} \},
  \end{array}
  \end{equation}
  a linear combination of the terms $z_{n-1}, z_n$, and $z_{n+1}$,
  with $z_{n-1}$ appearing with coefficient $1$.
  Similarly, if $n \geq 3$,
  \begin{equation}
    \begin{array}{lcl}
      (P_V A)^3 z_{n+1} &=&   P_V A (z_{n-1} - 2z_{n} + z_{n+1})    \\
      &=&  \{z_{n-2} - z_{n-1} - 2z_{n-1} + 2z_{n} + z_n - z_{n+1} \}  \\
      &=&  \{ z_{n-2} - 3z_{n-1} + 3z_n - z_{n+1} \}.
  \end{array}
  \end{equation}
  By iterative composition, $(P_V A)^m z_{n+1}$ is linear combination of the terms
  $z_p$ for $n-m+1 \leq p \leq n+1$, 
  with $z_{n-m+1}$ appearing with coefficient $1$,
  as long as $n-m+1 \geq 0$.
  Since $\dom A$ is a linear subspace of $X$,
  $(P_V A)^m z_{n+1} \subset \dom A$ if $n \geq m$.
  However, if $n+1 = m$ the single point in $(P_V A)^m z_{n+1}$ is not in $\dom A$
  since $z_0 = x \notin \dom A$.
\end{proof}

For any linear relation $A : X \to 2^X$ where $\dom A$ is not closed, sequences like those
in Proposition~\ref{p:curiousLRresult} are plentiful.
Every point $x \in \overline{\dom A}$
such that $x \notin \dom A$, including for instance
the points $\lambda x$ for $\lambda > 0$, generates a different
sequence $(z_n)_{\nnn}$ using the method from the proof of Proposition~\ref{p:curiousLRresult}.

To explore these concepts, consider the following example.

\begin{example}
  \label{ex:growunboundedLRell2}
	Consider the infinite dimensional Hilbert space $\ell^2$, the space of infinite
	sequences $\mathbf{x} = (x_k)_{k \in \NN}$ such that $\sum_{k=1}^{+\infty} x_k^2 < +\infty$.
	Let $\mathbf{e}_k$ denote the $k$th standard unit vector 
	(the $k$th element in the sequence is $1$, and all other elements in the sequence are $0$).
  Define the single-valued monotone relation $A : \ell_2 \to \ell_2$
  defined for $x \in \dom A$ by
  \[
  A \mathbf{x} = A (\sum_{k=1}^{+\infty} x_k \mathbf{e}_{k}) := 
  \sum_{k=1}^{+\infty} k x_k \mathbf{e}_{k},
  \]
  where
  \[ 
  \dom A := \left\{
  x \in \ell_2 : 
  \exists N \in \NN \textrm{~s.t.~} x_k = 0 ~\forall k \geq N \right\}.
  \]
\end{example}

Considering the linear relation $A$ in the example above,
the point $\mathbf{x} := \sum_{k=1}^{+\infty} \frac{1}{k} \mathbf{e}_k$ is not in $\halo A$.
This is because the sequence $(\mathbf{y}_n)_{\nnn} \subset \dom A$ where $\mathbf{y}_n := \sum_{i=1}^n \frac{1}{2i} \mathbf{e}_i$ 
eventually violates (\ref{e:halo}) for any choice of $M > 0$ for a large enough $n$.
(Therefore we know that $A$ is not maximal monotone.)
However, the point $\mathbf{z} := \sum_{i=1}^{+\infty} \frac{1}{i^2} \mathbf{e}_i$ is in $\halo A$,
and $\gra A$ could be extended by the point $(\mathbf{z},\mathbf{x})$ and remain monotone.
Since $\mathbf{x} \in \overline{\dom A}$ but $\mathbf{x} \notin \halo A$,
yet $\mathbf{x} = A\mathbf{z}$ and $\mathbf{z} \in \halo A$, we have the beginning of a sequence
like those in Proposition~\ref{p:curiousLRresult} for any monotone extension of $A$
containing $(\mathbf{z},\mathbf{x})$ that is also a linear relation.

Finally, the following result is used later and appears in Proposition~4.6 in
\cite{HHBBorweinDecomp}.

\begin{proposition}[\cite{HHBBorweinDecomp}]
  \label{p:MMsymmsubgrads}
  Suppose that $A:X \to 2^X$ is a linear relation.
  Then $A$ is maximal monotone and symmetric if and only if 
  there exists a proper lower semicontinuous convex function $f : X \to \RR \bigcup \{+\infty\}$
  such that $A = \partial f$.
\end{proposition}

\section{Monotone classes of linear relations}
\label{s:LRPM}

The recent result for paramonotonicity and $3^*$-monotonicity below appears in 
\cite{HXLtoappear}.  

\begin{proposition}[\cite{HXLtoappear}]
  \label{p:HXLPM3*main}
  Suppose $A : X \to 2^X$ is a maximal monotone linear relation such that
  $\dom A$ and $\ran A_+$ are closed ($A_+$ is the symmetric part of $A$).
  Then, $A$ is $3^*$-monotone if and only if $A$ is paramonotone.
\end{proposition}

In this section we use a different approach to that used for Proposition~\ref{p:HXLPM3*main},
where we (while avoiding the use of the Fitzpatrick function)
obtain results that apply to all monotone operators regardless maximal monotonicity.
This is done by examining
the density of $\dom A$ rather than its closure, further extending these results.
First, we characterize paramonotonicity for linear relations with the following two facts.

\begin{fact}
  Suppose $A:X \to 2^X$ is a monotone linear relation.
  Then, $A$ is paramonotone if and only if
  for all $x \in X$
  \begin{equation}
    \langle x, Ax \rangle = 0 \Rightarrow Ax = A0.
    \label{e:LRpmchar}
  \end{equation}
\end{fact}
\begin{proof}
  Suppose that $A$ is paramonotone
  and that for some $x \in \dom A$, 
  $\langle x, Ax \rangle = 0$.
  Then, $\langle x - 0, Ax - A0 \rangle = 0$,
  since $A0 \subset (\dom A)^{\perp}$ (Proposition~\ref{p:LRdomain0}).
  Therefore, by paramonotonicity,
  every $x^* \in Ax$ is also in $A0$.
  By Fact~\ref{f:lineartraits} \ref{lineartraits3} and~\ref{lineartraits4},
  $Ax = A0$.\\
  Now, suppose that (\ref{e:LRpmchar}) holds for $A$ and that for some
  $(y,y^*), (z,z^*) \in \gra A$,
  \[\langle y - z, y^* - z^* \rangle = 0. \]
  Let $x = y - z$.  Since $A$ is a linear relation,
  $y^* - z^* \in Ax$, and so $\langle x, Ax \rangle = 0$.
  Therefore, $Ax = A0$, and so $y^* - z^* \in A0$
  and
  \[ y^* \in z^* + A0; \qquad -z^* \in -y^* + A0. \]
  By Fact~\ref{f:lineartraits} \ref{lineartraits1} and~\ref{lineartraits4},
  $-y^* + A0 = -Ay$.
  Hence $y^* \in Az$ and $z^* \in Ay$, so $A$ is paramonotone.
\end{proof}
  
\begin{fact}
  \label{f:LRequivA0}
  Suppose $A:X \to 2^X$ is a monotone linear relation,
  and let $x \in X$.
  Then, $Ax = A0$ if and only if $0 \in Ax$ and if $0 \in Ax$,
  then 
  $P_{A0^{\perp}} Ax = \{0\}$.
  If $A0$ is closed and $P_{A0^{\perp}} Ax = \{0\}$,
  then $0 \in Ax$.
\end{fact}
\begin{proof}
  Let $Ax = A0$.  Since $A0$ is a linear subspace of $X$
  (Fact~\ref{f:lineartraits} \ref{lineartraits3}),
  $0 \in Ax$. \\
  Let $0 \in Ax$.  Then, by 
  Fact~\ref{f:lineartraits} \ref{lineartraits4},
  $Ax = A0$.\\
  By Proposition~\ref{p:lineartraitA0perp},
  $P_{A0^{\perp}} Ax$ is a singleton, and since $0 \in A0^{\perp}$
  by the definition of a perpendicular set,
  $P_{A0^{\perp}} Ax = \{0\}$.\\
  Let $P_{A0^{\perp}} Ax = \{0\}$ and suppose that $A0$ is closed.
  Then, by Proposition~\ref{p:lineartraitA0perp},
  $0 \in Ax$.
\end{proof}

\begin{proposition}
	\label{p:linear3*isPM}
  Suppose $A:X \to 2^X$ is a monotone linear relation such that $\dom A$ 
  is dense in $A0^{\perp}$ and $A0$ is closed.
  If $A$ is $3^*$ monotone, then $A$ is also paramonotone.
\end{proposition}
\begin{proof}
Suppose that $A$ is not paramonotone,
so there exists an $x \in \dom A$ such that $\langle x , Ax \rangle = 0$ yet $A x \neq A0$.
Choose any $x^* \in Ax$, and let $x_0^* = P_{A0^{\perp}}x^*$.
By Fact~\ref{f:LRequivA0}, $x_0^* \neq 0$ since $A0$ is closed.
If $x_0^* \in \dom A$, let $w = \frac{1}{2} x_0^*$.
If $x_0^* \notin \dom A$, there is a sequence $(y_n)_{\nnn} \subset \dom A$ converging 
to $x_0^*$ since $\dom A$ is dense in $A0^{\perp}$.  
In this case, let $w = y_n$ for some $n$
such that 
\[ \langle w, Ax \rangle = \langle y_n, x_0^* \rangle \geq \frac{1}{2} \|x_0^*\|^2 . \]
Let $v = \lambda x$ for some $\lambda > 0$ and let $u=0$ so that
\[
\langle w - v, Av - Au  \rangle = \langle w - \lambda x, \lambda Ax \rangle \geq \frac{\lambda}{2} \|x_0^*\|^2
\]
which is unbounded with respect to $\lambda$.
Hence, $A$ is not $3^*$-monotone, yielding the contrapositive.
\end{proof}


We therefore obtain by a different method the following result from \cite{HXLtoappear}.

\begin{corollary}[\cite{HXLtoappear}]
  \label{c:linear3*isPM}
  If the linear relation $A : X \to 2^X$ is maximal monotone
  and $3^*$-monotone,
  then $A$ is paramonotone.
\end{corollary}
\begin{proof}
  Follows directly from Proposition~\ref{p:linear3*isPM} and Corollary~\ref{c:LRMMb55}.
\end{proof}

\begin{corollary}
  If the linear relation $A : X \to 2^X$ is $3^*$-monotone,
  then the operator $\tilde{A} : X \to 2^X$ defined by
  \begin{equation}
    \tilde{A} x := Ax + (\dom A)^{\perp}
    \label{e:LRextenddomA}
  \end{equation}
  is a linear relation and is a $3^*$-monotone extension of $A$ that is paramonotone.
\end{corollary}
\begin{proof}
  The operator $\tilde{A}$ is a linear relation since $A$ is a linear relation,
  since $\dom \tilde{A} = \dom A$, and since $(\dom A)^{\perp}$ is a linear subspace.
  (Recall that we are using the convention that $\emptyset + S = \emptyset$
  for any set $S$.)
  More specifically, for all $x,y \in \dom \tilde{A} = \dom A$ and for all $\lambda \in \RR$,
  \[ \lambda \tilde{A} x = \lambda Ax + \lambda (\dom A)^{\perp}
  \subset A(\lambda x) + (\dom A)^{\perp} = \tilde{A} (\lambda x), \]
  and
  \[ \tilde{A} x + \tilde{A} y = Ax + (\dom A)^{\perp} + Ay
  \subset A(x + y) + (\dom A)^{\perp} = \tilde{A} (x + y). \]
  By the definition of $(\dom A)^{\perp}$,
  for all $x,y,z \in \dom \tilde{A}$
  \[ \langle z - y, \tilde{A}y - \tilde{A}z \rangle
   = \langle z - y, Ay - Az \rangle. \]
  Therefore,
  $\tilde{A}$ is monotone and $3^*$-monotone because $A$ is monotone
  and $3^*$-monotone.
  Since by Proposition~\ref{p:LRdomain0}, $A0 \subset (\dom A)^{\perp}$,
  it follows from Fact~\ref{f:lineartraits} \ref{lineartraits4}
  that $\tilde{A}$ is a monotone extension of $A$
  and that $\tilde{A}0 = (\dom A)^{\perp}$.
  Therefore, $\tilde{A}0^{\perp} = \overline{\dom A}$,
  and so by Proposition~\ref{p:linear3*isPM} and since $\dom A = \dom \tilde{A}$,
  $\tilde{A}$ is paramonotone.
\end{proof}

If the linear relation $A$ from Proposition~\ref{p:linear3*isPM} is also a single valued bounded linear operator,
then Proposition~\ref{p:linear3*isPM} is a corollary to a stronger result from
\cite{BrezisHaraux76}.

\begin{proposition} \cite{BrezisHaraux76}
	\label{p:3*linearisalphainverseSM}
	Let $A:X \to X$ be a bounded monotone linear operator.  Then, $A$ is $3^*$-monotone if and only if
	there exists an $\alpha > 0$ such that
	\[ \langle x, Ax \rangle \geq \alpha \langle Ax, Ax \rangle = \alpha \|Ax\|^2\].
\end{proposition}

\begin{corollary}
  \label{c:blinear3*isPM}
	If $A:X \to X$ is a bounded linear $3^*$-monotone operator, then it is paramonotone.
\end{corollary}

However, there are $3^*$-monotone linear relations that are not paramonotone.

\begin{example} \label{ex:LR3*notPM}
  Let $X = \ell^2$ and define the operators $\tilde{A}, A : X : 2^X$ 
  by
  \begin{equation}
    \tilde{A}\mathbf{x} := \sum_{k=1}^{+\infty} x_{2k} \mathbf{e}_{2k}
    \label{e:LR3*notPMtildeA}
  \end{equation}
  and
  \begin{equation}
    A\mathbf{x} := x_1 \mathbf{u}
    \left( \sum_{k=1}^{\infty} \frac{1}{k} \mathbf{e}_{2k+1} \right)
    + \tilde{A}\mathbf{x} + A0
    \label{e:LR3*notPMA}
  \end{equation}
  where
  \begin{equation}
    \mathbf{u} := \left( \sum_{k=1}^{\infty} \frac{1}{k} \mathbf{e}_{2k+1} \right),
  \end{equation}
  \begin{equation}
  A0 := \left\{
\mathbf{x} \in \ell_2 : 
  \exists N \in \NN \textrm{~s.t.~} x_k = 0 \forall k \geq N \\ 
  \textrm{~and~} x_{2k + 1} = 0 \forall k \in \NN \right\},
    \label{e:LR3*notPMA0}
  \end{equation}
  and
  \begin{equation}
    \dom A = \dom \tilde{A} = \spann \{\mathbf{e}_1, \mathbf{e}_2, \mathbf{e}_4, \mathbf{e}_6, \ldots\}.
  \end{equation}
  Then, $A$ is a $3^*$-monotone linear relation, but it is not paramonotone.
\end{example}
\begin{proof}
  Both $A$ and $\tilde{A}$ are by definition linear relations.
  Note that $\tilde{A}$ is merely $\Id$ on $X$ with a domain reduction,
  Therefore, $\tilde{A}$ is $3^*$-monotone as it is a subgraph of $\Id$,
  which is $3^*$-monotone.
  Also, $A0$ is a dense subspace of $\spann \{\mathbf{e}_{2k+1} : k \in \NN\}$,
  and so $A0^{\perp} = \spann \{\mathbf{e}_{2k} : k \in \NN\}$.
  Therefore, $P_{A0^{\perp}} A\mathbf{x} = \tilde{A} \mathbf{x}$
  as $\mathbf{u} \in (\dom A)^{\perp}$.
  Since $\overline{A0} \subset (\dom A)^{\perp}$ (Proposition~\ref{p:LRdomain0}),
  for all $(\mathbf{x},\mathbf{x}^*), (\mathbf{y},\mathbf{y}^*), (\mathbf{z},\mathbf{z}^*) \in \gra A$,
  \[ \langle \mathbf{z} - \mathbf{y}, \mathbf{y}^* - \mathbf{x}^* \rangle
  = \langle \mathbf{z} - \mathbf{y}, P_{A0^{\perp}} \mathbf{y}^* - P_{A0^{\perp}} \mathbf{x}^* \rangle
  = \langle \mathbf{z} - \mathbf{y}, A\mathbf{y} - A\mathbf{x} \rangle, \]
  and so $A$ is also $3^*$-monotone.
  Now, 
  \[ A\mathbf{e}_1 = \mathbf{u} + A0 \not\subset A0, \]
  and so $A\mathbf{e}_1 \neq A0$. However, 
  $\langle \mathbf{e}_1, A\mathbf{e}_1 \rangle = \langle \mathbf{e}_1, \tilde{A}\mathbf{e}_1 \rangle = 0$.
  Therefore, $A$ is not paramonotone.
\end{proof}

\section{Monotone classes of linear operators}

By Proposition~\ref{p:linearMM},
monotone linear relations with full domain
are single-valued maximal monotone operators, and these operators correspond to \emph{linear operators}.
We consider linear operators henceforth in light of Proposition~\ref{p:LRdimreduction}, 
and examine their properties of monotonicity in $\RR^2$ and $\RR^n$.

The results of Sections~\ref{s:LR} and~\ref{s:LRPM} hold in their strongest form as
in $\RR^n$ all subspaces are closed.  Linear operators in $\RR^n$ are here identified with
their matrix representation in the standard basis.  Recall from
Proposition~\ref{p:MMsymmsubgrads} that symmetric linear operators are the
subdifferentials of a lower semicontinuous convex function.

\subsection{Monotone linear operators on $\RR^2$}

In this section we consider linear operators $A : \RR^2 \to \RR^2$,
which can be represented by the matrix
  \[ A = \left[ \begin{array}{cc} 
      a & c \\
      b & d 
  \end{array}\right]. \]
The operator $A$ so defined is monotone if and only if
$a + d \geq 0$ and $4ad \geq (b+c)^2$.
We consider some simple examples, examine their properties,
and provide some sufficient and necessary conditions for inclusion
within various monotone classes.

\begin{proposition}[$3$-cyclic monotone linear operators on $\RR^2$]
  \label{p:3cmlinearR2}
    If $A$ is $3$-cyclic monotone, then
    \begin{equation}
        0 \geq \max \{ |b|, |c|\} - a - d .
    \end{equation}
\end{proposition}
\begin{proof}
    Choose $x = (0,0)$,
    $y = (1,0)$, and $z = (0,1)$;
    let $x^* = Ax = (0,0)$, 
    $y^* = Ay = (a,b)$, and $z^* = Az = (c,d)$.
    If the mapping associated with $A$ is
    $3$ cyclic monotone then
    \begin{eqnarray*}
      0 &\leq&
        \langle x-y, x^* \rangle 
        + \langle y-z, y^* \rangle
        + \langle z-x, z^* \rangle    \\
        &=& \langle (1,-1), (a,b) \rangle
        + \langle (0,1), (c,d) \rangle \\
        &=& a + d - b.
    \end{eqnarray*}
    Similarly, by choosing different $y$ and $z$,
    the following conditions are also necessary for any matrix $A$ as defined
    above:
    \begin{equation}
        0 \geq 
        \left\{
        \begin{array}{cc}
            b - a - d, & y = (1,0), z = (0,1), \\
            -b - a - d, & y = (-1,0), z = (0,1), \\
            c - a - d, & y = (0,1), z = (1,0), \\
            -c -a -d, & y = (0,-1), z = (1,0).
        \end{array}
        \right. 
    \end{equation}
    In all cases, $x = (0,0)$.
\end{proof}

There are many monotone linear operators in $\RR^2$ that are not
$3$-cyclic monotone, and furthermore
Examples~\ref{ex:linear2Da}~and~\ref{ex:linear2D} below demonstrate that
$3$-cyclic monotonicity does not follow from strict and maximal monotonicity.

\begin{example}  
    \label{ex:linear2Da}
    Consider the monotone linear operator $\tilde{R} : \RR^2 \to \RR^2$
    defined by
    \begin{equation} \label{def:tildeR} 
        \tilde{R} = \left[ \begin{array}{cc} 
            1 & -2 \\ 
            3 & 1
        \end{array}\right] .
    \end{equation}
    $\tilde{R}$ violates the necessary conditions for $3$-cyclic monotonicity
    since $b - a -d > 0$ and satisfies the monotonicity conditions
    $(a + d) \geq 0$ and $4ad \geq (b+c)^2$,
    using the format
    $\tilde{R} = \left[ \begin{array}{cc} 
        a & c \\
        b & d 
    \end{array}\right]$
    above.
    Note that $\langle x, \tilde{R}x \rangle = 0$ implies that $x = 0$,
    so $\tilde{R}$ is strictly monotone and therefore paramonotone.
    Hence, by Proposition~\ref{p:linearPMis3*isPM}, $\tilde{R}$ is also $3^*$-monotone.
    $\tilde{R}$ is maximal monotone by Proposition~\ref{p:linearMM}.
\end{example}

\begin{example}
  \label{ex:linear2D}
  Consider the rotation operator $R_{\theta} : \RR^2 \to \RR^2$
  with matrix representation
    \begin{equation}
        R_{\theta} = \left[ \begin{array}{cc} 
            \cos(\theta) & -\sin(\theta) \\
            \sin(\theta) & \cos(\theta) 
        \end{array}\right].
        \label{def:rotationmatrix}
    \end{equation}
    Note that $R_{\theta}$ is monotone if and only if $|\theta| \leq \pi/2$,
    since this is precisely when $\cos(\theta) \geq 0$.
    In this range, $R_{\theta}$ is maximal monotone by Proposition~\ref{p:linearMM}.\\
    Now, $R_{\theta}$ is $3$-cyclic monotone if and only if
    $|\theta| < \pi/3$ by Fact~\ref{f:thetacyclic} below. \\
    Therefore, for any $\theta \in ]\pi/3, \pi/2[$,
    $R_{\theta}$ is maximal monotone and strictly monotone, but not $3$-cyclic monotone.\\
    Now, $\langle x, R_{\theta}x \rangle = 0$ implies that $x = 0$ unless $\theta = \pi/2$.
    Therefore, $R_{\theta}$ is strictly monotone and hence paramonotone when $|\theta| < \pi/2$.
    By Proposition~\ref{p:linearPMis3*isPM}, $R_{\theta}$ is $3^*$-monotone as well when
    $|\theta| < \pi/2$. 
    When $\theta=\pi/2$, $R_\theta$ is not paramonotone, and therefore it is neither strictly monotone,
    nor, by Proposition~\ref{p:linear3*isPM}, is it $3^*$-monotone.
\end{example}

By the following fact, $\RR^2$ is
large enough to contain distinct instances of $n$-cyclic monotone operators for $n \geq 2$.

\begin{fact}[\cite{Bauschke44} Proposition~7.1]
    \label{f:thetacyclic}
    Let $n \in \{2,3,\ldots\}$. Then $R_{\theta}$ is
    $n$-cyclic monotone if and only if $|\theta| \in [0, \pi / n]$.
\end{fact}
\begin{proof} See Example~4.6 in \cite{Bauschke44} for a detailed proof.
\end{proof}

\begin{example}
  \label{ex:zero2nd}
  The orthogonal projection $A: \RR^2 \to \RR^2$ defined by $A(x_1,x_2) := (x_1,0)$
  is maximal monotone, paramonotone, $3$-cyclic monotone, and $3^*$-monotone.
\end{example}
\begin{proof}
  Using the notation of Section~\ref{s:productSpace}, we have that $A = Id \times \mathbf{0}$, where $\mathbf{0}: \RR \to \RR$ is the zero operator,
  and $\Id : \RR \to \RR$ is the identity.
  The $\mathbf{0}$ operator is maximal monotone, paramonotone, $3$-cyclic monotone, and $3^*$ monotone,
  as is $\Id$, which is also strictly monotone, while $\mathbf{0}$ is not.
  The properties of $A$ follow directly from the results in Section~\ref{s:productSpace}.
\end{proof}

Finally, paramonotone linear operators in $\RR^2$ are further restricted to be either strictly monotone or symmetric.

\begin{proposition}
    \label{p:SMonR2}
    A linear operator $A:\RR^2 \to \RR^2$ is paramonotone if and only if it is 
    strictly monotone or symmetric.
\end{proposition}
\begin{proof}
  Strictly monotone operators and symmetric linear operators are trivially paramonotone
  by definition and Fact~\ref{f:LandP} respectively.
  It remains to show that these are the only two possibilities.
  Assuming then that $A$ is paramonotone, consider the general case, 
  $A = \left[ \begin{array}{cc} 
      a & c \\
      b & d 
  \end{array}\right]$
  and
  $A_+ = \left[ \begin{array}{cc} 
    a & \frac{b+c}{2} \\
    \frac{b+c}{2} & d 
  \end{array}\right]$.
	If $\ker(A_+) = \{0\}$, then $A$ is strictly monotone by
  Fact~\ref{f:LandP}.
	If $\ker(A_+) \neq \{0\}$ then
  by Fact~\ref{f:LandP}
  $\ker(A_+) \subseteq \ker(A)$,
  and so $\ker(A) \neq \{0\}$, from which $\det(A)=0$ and $ad=bc$.
  Hence, since $\det(A_+)$ = 0, $4bc = (b+c)^2$, so $(b-c)^2=0$ and $b=c$. 
  Therefore $A$ is symmetric.
\end{proof}

\begin{remark}
	\label{r:linearPMnotSMimpliesCM}
  The only paramonotone linear operators in $\RR^2$
  that are not strictly monotone are the symmetric linear operators
  $A := \left[ \begin{array}{cc}
      a & b \\
      b & \frac{b^2}{a}
  \end{array} \right]
  $
  for $a > 0$ and $b \in \RR$ and
  the zero operator $x \mapsto (0,0)$.
	By Proposition~\ref{p:MMsymmsubgrads}, since both examples of $A$ are symmetric linear
	operators, they are also maximal monotone and maximal cyclical monotone,
  as they are subdifferentials of proper lower semicontinuous convex functions.
\end{remark}

All relationships between the classes of monotone linear operators 
in $\RR^2$ are now known completely and are summarized 
in Table~\ref{table:linear2D}.
Recall that all monotone linear operators are assumed to have full domain 
and are therefore maximal monotone by Proposition~\ref{p:linearMM}.

\begin{table}[!h]
  \caption{Monotone linear operators on $\RR^2$:  monotone class relationships.}
    \centering
    \begin{tabular}{c c c c | c l}
    \label{table:linear2D}
        PM & SM & 3CM & 3* \\
        \hline
        0&0&0&0& $\exists$ & Example~\ref{ex:linear2D} ($R_{\pi/2}$) \\
				0&*&*&1& $\emptyset$ & Proposition~\ref{p:linear3*isPM} \\
        *&*&1&0& $\emptyset$ & Fact~\ref{f:3CMis3*} \\
        0&*&1&*& $\emptyset$ & Proposition~\ref{p:3-2cyclic} \\
        0&1&*&*& $\emptyset$ & Fact~\ref{f:SMisPM} \\
				1&*&*&0& $\emptyset$ & Proposition~\ref{p:linearPMis3*isPM} \\
				1&0&0&*& $\emptyset$ & Remark~\ref{r:linearPMnotSMimpliesCM} \\
        1&0&1&1& $\exists$ & Example~\ref{ex:zero2nd} ($A(x_1,x_2) := (x_1,0)$) \\ 
				1&1&0&1& $\exists$ & Example~\ref{ex:linear2D} ($R_{\theta}$, $\pi/2 > |\theta| > \pi/3$)\\
        1&1&1&1& $\exists$ & $Id$ \\ 
        \hline
    \end{tabular}
    \\
\begin{tabular}{l}
Where: \\
'PM' represents paramonotone, \\
'SM' represents strictly monotone, \\
'3CM' represents $3$-cyclic monotone, and \\
1 represents that the property is present \\
0 represents an absence of that property \\
* represents that both 0/1 are covered by the example/result. \\
$\exists$ represents that an example with these properties exists. \\
$\emptyset$ represents that this combination of properties is impossible. 
\end{tabular}
\end{table}

\subsection{Linear operators on $\RR^n$}
On $\RR^n$ the restriction that linear operators are single-valued is redundant as this also follows from having full domain. 

\begin{proposition}
  A single valued monotone linear relation $A : \RR^n \to \RR^n$
  is maximal monotone if and only if $\dom A = \RR^n$.
\end{proposition}
\begin{proof}
  In $\RR^n$, all subspaces are closed, and so by Corollary~\ref{c:LRMMb55},
  any maximal monotone single valued linear relations have full domain.
  The converse follows from Proposition~\ref{p:linearMM}.
\end{proof}

Since linear operators are maximal monotone,
the following result follows from Proposition~\ref{p:HXLPM3*main}.

\begin{proposition}[\cite{Bauschke44}]
  \label{p:linearPMis3*isPM}
  Given a monotone linear operator $A : \RR^n \to \RR^n$,
  $A$ is $3^*$-monotone if and only if $A$ is paramonotone.
\end{proposition}

In the following fact, we denote by $A_+ := \frac{1}{2} (A + A^*)$ the symmetric part of a linear operator $A: \RR^n \to \RR^n$, and by $\ker A := \{x \in \RR^n : Ax = 0 \}$ is the kernel of $A$.

\begin{fact}[\cite{Iusem98}]
  \label{f:LandP}
  Let $A: \RR^n \to \RR^n$ be a linear operator.
  Then $A$ is paramonotone if and only if $A$ is monotone and $\ker(A_+) \subseteq \ker(A)$.
\end{fact}

In Remark~\ref{r:linearPMnotSMimpliesCM}
we noted that the converse of Proposition~\ref{p:3-2cyclic}
holds for
monotone linear operators that are not strictly monotone operators on $\RR^2$.
We now demonstrate that 
this result does not generalize to $\RR^3$.

\begin{example} \label{ex:linearcexample}
    Let $T : \RR^3 \to \RR^3$ be the linear operator defined by
    \begin{equation}
        \label{e:linearcexample}
        T x := \left[ \begin{array}{ccc}
            1 & -2 & 1 \\
            3 & 1 & 3 \\
            1 & -2 & 1
        \end{array} \right] x
    \end{equation}
    $T$ is paramonotone and maximal monotone, but not strictly monotone. 
    $T$ is not $3$-cyclic monotone, but is $3^*$-monotone.
\end{example}
\begin{proof}
    The symmetric part of $T$ is
    \begin{equation*}
        T_{+} := \left[ \begin{array}{ccc}
            1 & 1/2 & 1 \\
            1/2 & 1 & 1/2 \\
            1 & 1/2 & 1 
        \end{array} \right]
    \end{equation*}
    Since the eigenvalues of $T_+$, consisting of 
    $\{0, \frac{1}{2}(3 + \sqrt{3}), \frac{1}{2}(3 - \sqrt{3})\}$,
    are nonnegative,
    $T_+$ is positive semidefinite, hence monotone, and so
    $T$ is monotone.\\
    An elementary calculation yields that
    $\ker T_+ = \{ t (-1, 0, 1) : t \in \RR \}$.
    Clearly, $\ker T = \ker T_+$, so
    by Fact~\ref{f:LandP},
    $T$ is paramonotone.
    However, $T$ is
    not strictly monotone since the kernel contains more than the zero element.
    \\
    $T$ is maximal monotone since it is linear and has full domain (Proposition~\ref{p:linearMM}).
    Finally, $T$ is not $3$-cyclic monotone since the points
    $(0,0,0), (1,0,0),$ and $(0,1,0)$ do not satisfy the defining condition
    (\ref{e:3cyclicgeq0}).  
    (For a shortcut, call to mind Example~\ref{ex:linear2Da} and Proposition~\ref{p:3cmlinearR2}.)
		Finally, since $T$ is a linear operator in $\RR^3$ that is paramonotone, it is 
		$3^*$-monotone by Proposition~\ref{p:linearPMis3*isPM}.
\end{proof}

\subsection{Monotone linear operators in infinite dimensions}

Recall from Proposition~\ref{p:linearPMis3*isPM} that linear paramonotone operators
on $\RR^n$ are $3^*$ monotone.  
Example~\ref{ex:linearPMnot3*} below
demonstrates that larger spaces are more permissive.
A similar example appears in \cite{HXLtoappear}.

\begin{example}
	\label{ex:linearPMnot3*}
	Let $\theta_k := \pi/2 - 1/k^4$ and
	let $A: \ell^2 \to \ell^2$ be the linear operator defined by
	\begin{equation}
		A \mathbf{x} \mapsto \sum_{k=1}^{+\infty} 
		\left( \cos(\theta_k) x_{2k-1} - \sin(\theta_k)x_{2k} \right) \mathbf{e}_{2k-1}
		+ \left( \sin(\theta_k) x_{2k-1} + \cos(\theta_k) x_{2k} \right) \mathbf{e}_{2k},
		\label{e:linearPMnot3*}
	\end{equation}
  The structure of $A$ is such that every $\mathbf{x}^* = A\bx$ obeys
  \begin{equation}
    \label{e:linearPMnot3*asR}
    \left[ \begin{array}{c} x^*_{2k-1} \\ x^*_{2k} \end{array} \right]
      = R_{\theta_k} 
    \left[ \begin{array}{c} x_{2k-1} \\ x_{2k} \end{array} \right]
  \end{equation}
  for all $\mathbf{x} \in \ell^2$ and $k \in \NN$,
  where $R_{\theta_k}$ is the rotation matrix as introduced in Example~\ref{ex:linear2D}.
  $A$ is strictly monotone and maximal monotone, but not $3^*$-monotone.
	It follows that $A$ is also paramonotone but not $3$-cyclic monotone.
\end{example}
\begin{proof}
  The monotonicity of $T$ is evident from (\ref{e:linearPMnot3*asR}).
	Suppose that $\mathbf{x} \in \ell^2$ is such that
	$\langle \mathbf{x}, A\mathbf{x} \rangle = 0$.
	Now, 
	\[
	\langle \mathbf{x}, A\mathbf{x} \rangle
	= \sum_{k=1}^{+\infty} \cos(\theta_k) (x_{2k-1}^2 + x_{2k}^2)
	\]
	is equal to zero if and only if $\mathbf{x} = \mathbf{0}$,
	and so $A$ is strictly monotone.\\
	By Proposition~\ref{p:linearMM}, $A$ is maximal monotone since it is linear and has full domain.\\
	Let $\mathbf{x} =\mathbf{0}$, so that $A\mathbf{x} = \mathbf{0}$,
	and let $\mathbf{z} = \sum_{k=1}^{+\infty} \frac{1}{k} (\mathbf{e}_{2k-1} + \mathbf{e}_{2k})$.
	Define a sequence $\mathbf{y}_n \in \ell^2$ by
	$\mathbf{y}_n := n^2 \mathbf{e}_{2n-1}$,
	and so 
	$A \mathbf{y}_n = n^2 \cos(\theta_n) \mathbf{e}_{2n-1} + n^2 \sin(\theta_n) \mathbf{e}_{2n}$.
	For all $n$, $0 < \cos(\theta_n) \leq 1/n^4$,
	and from the Taylor's series $\sin(\theta_n) \geq 1 - 1/(2n^8)$
	for all large $n$.
	Considering the inequality related to $3^*$-monotonicity, we have
	\begin{equation}
		\begin{array}{lll}
      \langle \mathbf{z} - \mathbf{y}_n, A\mathbf{y}_n - A\mathbf{x}  \rangle
			&=& n \left( \cos(\theta_n) + \sin(\theta_n) \right) - n^4 \cos(\theta_n) \\
			&\geq& n(0 + 1 - 1/(2n^8)) - 1 \\
			&\to& +\infty, \qquad \textrm{as~} n \to +\infty,
		\end{array}
	\end{equation}
	and so $A$ fails to be $3^*$-monotone.\\
\end{proof}

\begin{remark}
	\label{r:linearPMnot3*notSM}
	The operator $A$ from Example~\ref{ex:linearPMnot3*} 
  can be modified to lose 
	its strict monotonicity property by 
	using the zero function $\mathbf{0}:\RR \to \RR$
	as a prefactor in the product space, yielding $T = \mathbf{0} \times A$.
  In this manner,
	\begin{equation}
		T \mathbf{x} := \sum_{k=1}^{+\infty} \left[
    \begin{array}{l}
		\left( \cos(\theta_k) x_{2k} - \sin(\theta_k)x_{2k+1} \right) \mathbf{e}_{2k} \\
		\qquad + \left( \sin(\theta_k) x_{2k} + \cos(\theta_k) x_{2k+1} \right) \mathbf{e}_{2k+1}.
  \end{array} \right]
	\end{equation}
\end{remark}
\begin{proof}
  The Hilbert space $\ell^2$ can be written as a product space $\ell^2 = \RR \times \ell^2$.
  More precisely, all of these spaces can be embedded in the larger space $\ell^2(\mathcal{Z})$
  with standard unit vectors $e_i$, where $i \in \mathcal{Z}$.
  In this setting $\ell^2 = \textrm{span} \{e_i : i \in \NN\}$,
  and let $V_0 = \textrm{span} \{e_0\}$
  so that $\ell^2(\NN \bigcup \{0\}) = V_0 \times \ell^2$.
  Let $T = \mathbf{0} \times A$, where $A$ is the linear operator from Example~\ref{ex:linearPMnot3*}.
  The operator $\mathbf{0}: V_0 \to V_0$ is paramonotone, maximal monotone,
  $3$-cyclic monotone, and $3^*$-monotone, but not strictly monotone on $\RR$.
  The operator $A : \ell^2 \to \ell^2$ from Example~\ref{ex:linearPMnot3*}
  is strictly monotone and maximal monotone, but not $3^*$-monotone.
  Therefore, by the results of Section~\ref{s:productSpace},
  $T:= \mathbf{0} \times A$ is
  paramonotone and maximal monotone, and fails to be
  strictly monotone or $3^*$-monotone.
\end{proof}

Note that all linear operators are assumed to have full domain and are therefore maximal monotone by Proposition~\ref{p:linearMM}.
Also, if a linear operator fails to be paramonotone, it fails to be $3^*$-monotone
and $3$-cyclic monotone as well.  The results for linear operators in a Hilbert space can now be summarized as in Table~\ref{table:linear} below.

\begin{table}[!h]
    \caption{Monotone linear operators:  monotone class relationships}
    \centering
    \begin{tabular}{c c c c | c l}
    \label{table:linear}
        PM & SM & 3CM & 3* \\
        \hline
        0&0&0&0& $\exists$ & $R_{\pi/2}$ \\
				0&*&*&1& $\emptyset$ & Proposition~\ref{p:linear3*isPM} \\
        *&*&1&0& $\emptyset$ & Fact~\ref{f:3CMis3*} \\
        0&*&1&*& $\emptyset$ & Proposition~\ref{p:3-2cyclic} \\
        0&1&*&*& $\emptyset$ & Fact~\ref{f:SMisPM} \\
				1&0&0&0& $\exists$ & Remark~\ref{r:linearPMnot3*notSM} \\
				1&0&0&1& $\exists$ & Example~\ref{ex:linearcexample} \\
        1&0&1&1& $\exists$ & $\mathbf{0}$ \\ 
				1&1&0&0& $\exists$ & Example~\ref{ex:linearPMnot3*} \\
				1&1&0&1& $\exists$ & Example~\ref{ex:linear2D} ($R_{\theta}$, $\pi/2 > |\theta| > \pi/3$)\\
        1&1&1&1& $\exists$ & $Id$ \\ 
        \hline
    \end{tabular}
    \\
\begin{tabular}{l}
Where: \\
'PM' represents paramonotone, \\
'SM' represents strictly monotone, \\
'3CM' represents $3$-cyclic monotone, and \\
1 represents that the property is present \\
0 represents an absence of that property \\
* represents that both 0/1 are covered by the example/result. \\
$\exists$ represents that an example with these properties exists. \\
$\emptyset$ represents that this combination of properties is impossible. 
\end{tabular}
\end{table}

\section{Summary}
\label{s:Summary}

By bringing together existing works and new results, the relationship
between the five classes of monotone operator considered, that is maximal,
para, $3^*$-, $3-$cyclic, and strictly monotone operators, is now fully
understood in $\RR^2$, $\RR^n$ and in general Hilbert spaces.
The results of Section~\ref{s:LR}, particularly Proposition~\ref{p:LRdimreduction},
allows these results to be extended to linear relations.
Furthermore, results of Section~\ref{s:productSpace} can be used to generate
further linear operators belong to these classes, and can be used to determine
the monotone classes to which an operator belongs given a block-diagonal form by examining its composite blocks.
The following Venn diagrams summarize the relationships between these five classes of monotone operator.

\begin{figure}[!h]
  \begin{center}
    \epsfig{file=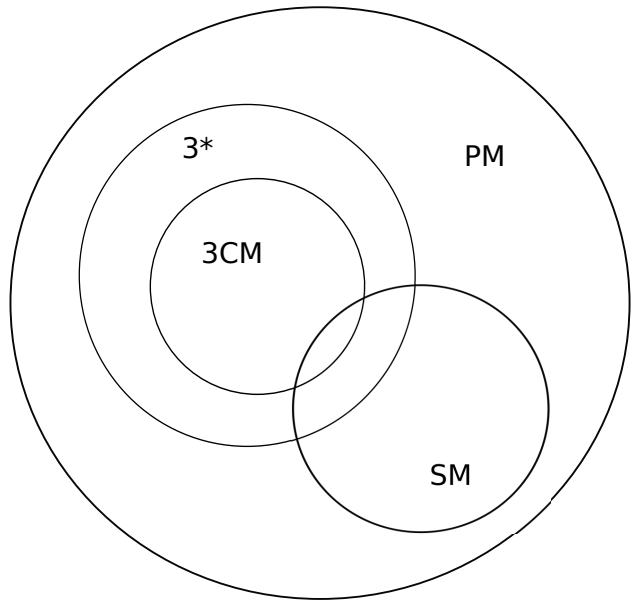, width=0.5\textwidth}
  \end{center}
  \caption{Monotone linear operators: monotone class relationships.
  PM = paramonotone, SM = strictly monotone, 3CM = $3$
  cyclic monotone, 3* = $3^*$-monotone.}
  \label{fig:4compareL}
\end{figure}

\begin{figure}[!h]
  \begin{center}
    \epsfig{file=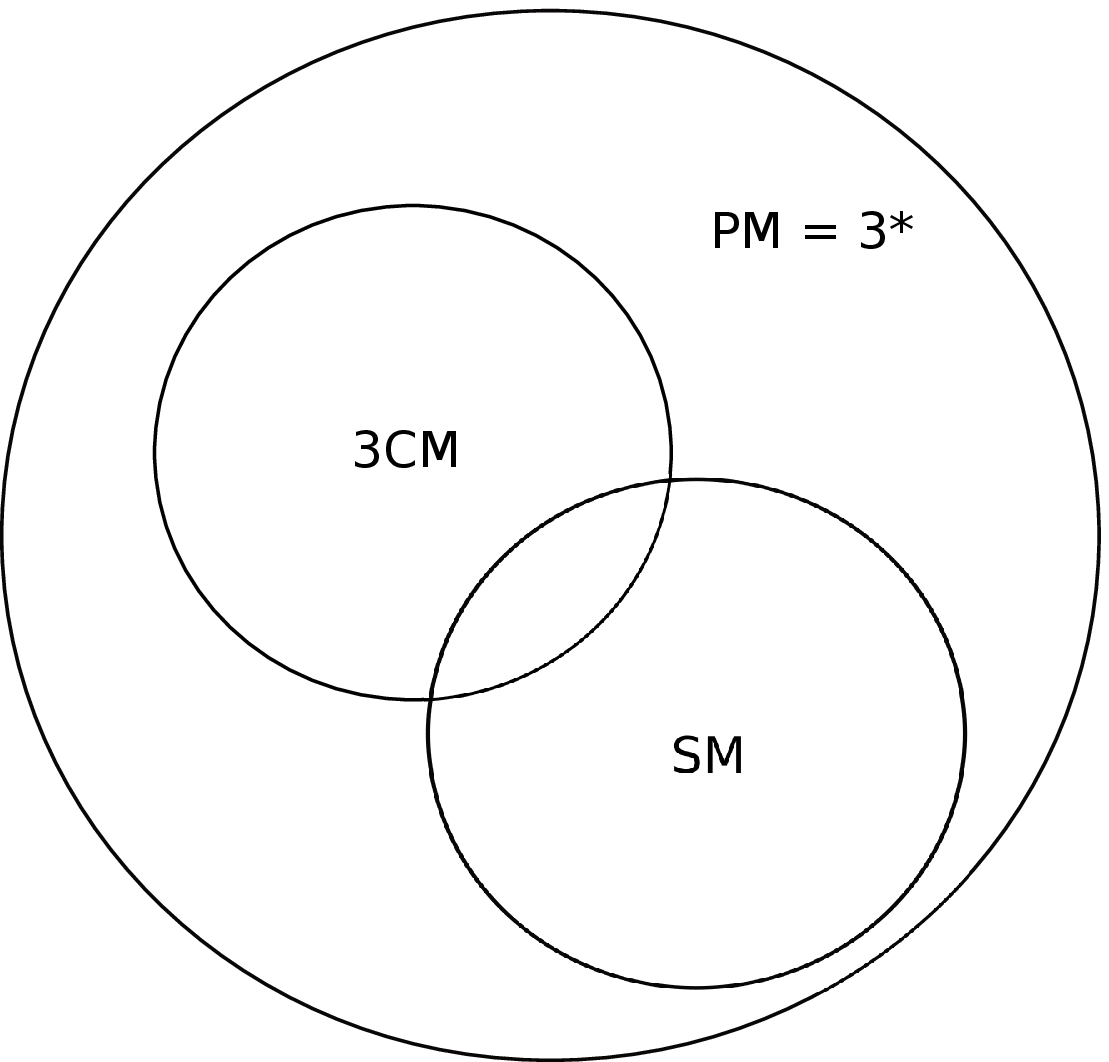, width=0.5\textwidth}
  \end{center}
  \caption{Monotone linear operators on $\RR^n$: monotone class relationships.
  PM = paramonotone, SM = strictly monotone, 3CM = $3$
  cyclic monotone, 3* = $3^*$-monotone.}
  \label{fig:4compareLRn}
\end{figure}

\begin{figure}[!h]
  \begin{center}
    \epsfig{file=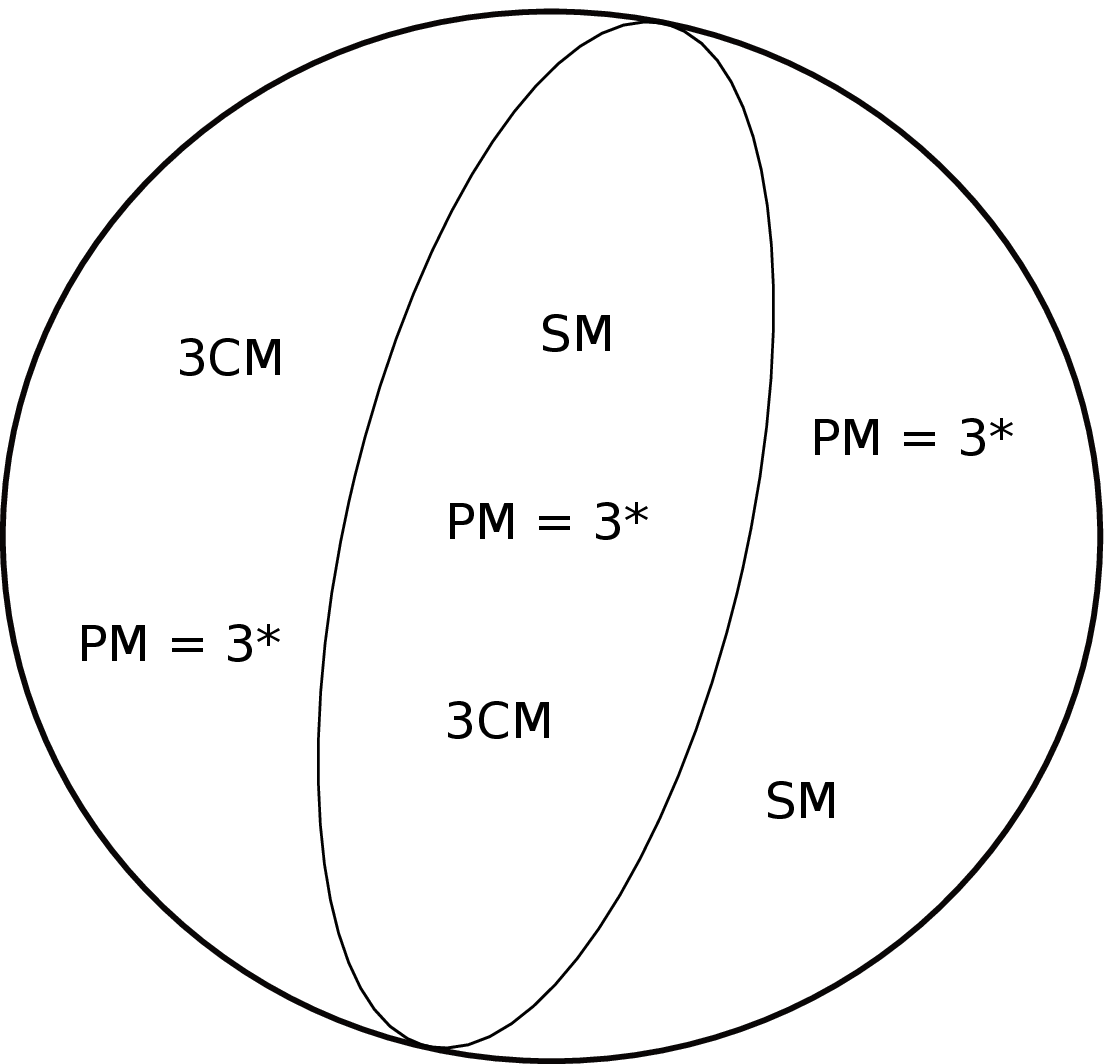, width=0.5\textwidth}
  \end{center}
  \caption{Monotone linear operators on $\RR^2$: monotone class relationships.
  PM = paramonotone, SM = strictly monotone, 3CM = $3$
  cyclic monotone, 3* = $3^*$-monotone.}
  \label{fig:4compareLR2}
\end{figure}

Further results for nonlinear operators involving these five
classes will appear shortly.  Some of those results and
the results of this paper were presented at a meeting of the Canadian Mathematical
Society in Vancouver, Canada on December 4, 2010.

\clearpage 

\bibliographystyle{plain}
\bibliography{bibtex}{}

\end{document}